\documentclass[12pt]{amsart}

\usepackage[left=0.7in, right=0.7in, top=1in, bottom=1in]{geometry}

\usepackage{amssymb}
\usepackage{amsfonts}
\usepackage{amsthm}

\usepackage{graphicx}
\usepackage{epstopdf}
\DeclareGraphicsExtensions{.pdf,.png,.jpg,.eps}

\usepackage{mathtools}
\usepackage{algpseudocode}
\usepackage{algorithm}
\usepackage{enumitem}
\usepackage{ulem}
\usepackage{color}
\usepackage{hyperref}
\usepackage{cleveref}
\usepackage[font=small, labelfont=bf]{caption}

\hypersetup{
  pdftitle={Support Graph Preconditioners for Off-Lattice Cell-Based Models},
  pdfauthor={Justin Steinman and Andreas Buttensch\"on},
  colorlinks=true,
  linkcolor=blue,
  citecolor=blue,
  urlcolor=blue
}

\theoremstyle{plain}
\newtheorem{theorem}{Theorem}[section]
\newtheorem{lemma}[theorem]{Lemma}

\theoremstyle{definition}
\newtheorem{definition}[theorem]{Definition}

\theoremstyle{remark}

\newtheorem{cor}[theorem]{Corollary}

\crefname{hypothesis}{Hypothesis}{Hypotheses}

\DeclareMathOperator{\spn}{span}
\newcommand{\R}{\mathbb{R}}
\renewcommand{\vec}[1]{\ensuremath{\mathbf{#1}}}
\newcommand{\mat}[1]{\ensuremath{\uline{#1}}}
\newcommand{\lb}{\left(}
\newcommand{\rb}{\right)}
\DeclarePairedDelimiter{\norm}{\lVert}{\rVert}

\title{Support Graph Preconditioners for Off-Lattice Cell-Based Models}

\author{Justin Steinman}
\address{Department of Mathematics and Statistics, University of Massachusetts Amherst, Amherst, MA}
\email{jsteinman@umass.edu}

\author{Andreas Buttensch\"{o}n}
\address{Department of Mathematics and Statistics, University of Massachusetts Amherst, Amherst, MA}
\email{abuttenschoe@umass.edu}


\begin{document}

\begin{abstract}
  Off-lattice agent-based models (or cell-based models) of multicellular systems are increasingly used to create in-silico models of in-vitro and in-vivo experimental setups of cells and tissues, such as cancer spheroids, neural crest cell migration, and liver lobules. These applications, which simulate thousands to millions of cells, require robust and efficient numerical methods. At their core, these models necessitate the solution of a large friction-dominated equation of motion, resulting in a sparse, symmetric, and positive definite matrix equation. The conjugate gradient method is employed to solve this problem, but this requires a good preconditioner for optimal performance. In this study, we develop a graph-based preconditioning strategy that can be easily implemented in such agent-based models. Our approach centers on extending support graph preconditioners to block-structured matrices. We prove asymptotic bounds on the condition number of these preconditioned friction matrices. We then benchmark the conjugate gradient method with our support graph preconditioners and compare its performance to other common preconditioning strategies.
\end{abstract}

\maketitle

\section{Introduction}

Agent-based models that simulate individual entities such as humans, animals, or biological cells are an indispensable tool for studying emergent behaviors in complex systems.
Over the last few decades, biomedical research has adopted agent-based models to develop
digital-twins of in-vitro and in-vivo experiments on cell cultures and tissues
\cite{Buttenschon2020}. To capture a wide variety of applications and research questions,
many different agent-based models have been developed. Broadly, we categorize them into
lattice-based (e.g.\ Cellular Automata \cite{hadeler2017cellular} or Cellular Potts models \cite{graner1992simulation}) and off-lattice models \cite{VANLIEDEKERKE2018}. These different
models each have advantages and disadvantages. For an overview, we refer the reader to the
review \cite{Van_Liedekerke2015-fm}. Here, we
focus on off-lattice models closely related to colloidal physics \cite{Drasdo2005-xf}.
In these models, cells are approximated by elastic spheroids \cite{Drasdo2005-cs,Ghaffarizadeh2018},
capsules \cite{MesenchymalCells}, ellipsoids \cite{Palsson2008}, or surfaces of triangulated
meshes \cite{VANLIEDEKERKE2018}.
The applications of these models are highly varied, including slime-mold
aggregation \cite{Palsson2008}, cancer growth and migration \cite{knutsdottir20163,macklin2012patient},
cancer monolayers and spheroids \cite{Drasdo2005-cs,van2019quantitative},
liver lobules \cite{hoehme2010prediction}, and neural crest cells \cite{mclennan2015neural}.

\begin{figure}[!ht]\centering
    \includegraphics[width=\textwidth]{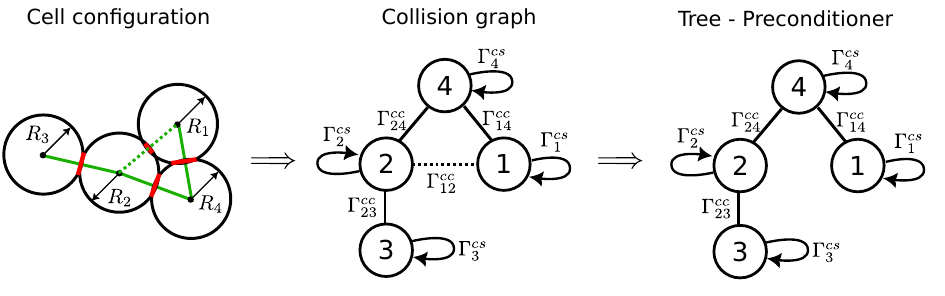}
    \caption{Graphical overview of our preconditioner construction.
    The figure illustrates the step-by-step construction of our proposed preconditioner,
    proceeding from left to right.
    A) Agent configuration. The initial setup showing individual agents (cells) in a spatial
    arrangement.
    B) The collision graph. Using collision detection algorithms, we construct a graph
    where nodes represent agents and edges represent collisions or friction interactions
    between them. The friction matrix is the graph Laplacian of the collision graph.
    C) The maximum spanning tree. Using Prim's algorithm, we construct a maximum spanning
    tree from the collision graph. The graph Laplacian of this tree, is used as the
    preconditioner.
    }\label{fig:graph}
\end{figure}

A typical cell configuration of spherical cells with radii $R_i$
is shown in Figure~\ref{fig:graph}. Advancing the simulation from $t \to t + \Delta t$
requires solving the overdamped equation of motion
$\Gamma \vec{v} = \vec{F}$, where $\Gamma$ is the
friction matrix, $\vec{v}$ the cells' velocities, and $\vec{F}$ the forces.
The matrix $\Gamma$ is block-structured, symmetric, and positive definite
because the individual $3 \times 3$ friction matrices are symmetric and positive definite \cite{VANLIEDEKERKE2018}.
There is a nonzero off-diagonal block in $\Gamma$ at position $(i, j)$ when cells
$i$ and $j$ are in contact.
We ultimately solve this large linear system using the conjugate gradient method.

Profiling our simulation software used in
\cite{MesenchymalCells,VANLIEDEKERKE2018,van2019quantitative} shows
that solving the equation of motion is often the most time-intensive step.
We hypothesize that this is because no good preconditioners have been identified or developed so far.
From an implementation point of view, it is convenient to implement the linear
algebra methods in a matrix-free manner. This means that ``off the shelf'' preconditioning
techniques are difficult to use or adapt. Further, the condition number of the friction
matrix $\Gamma$ is determined by the cells' free surface area. This means that the condition
number of the friction matrix varies during a given simulation.
Additionally, the sparsity structure of $\Gamma$ is dynamic because it encodes
interacting cell pairs. This makes selecting a preconditioner difficult.
Here, we solve this by using the collision graph constructed during the collision detection phase
as our central data structure instead of the usual sparse matrix implementations.

These observations, together with the increasing role agent-based models play in
biomedical research, motivate our work. Our goal is to develop and benchmark
a preconditioning strategy for the friction matrix $\Gamma$ that is easily implemented
in a matrix-free manner, and reduces the required computational time for solving the linear system.

\subsection{Support Graph Preconditioners}

The computational cost per iteration of the conjugate gradient method is dominated by the matrix-vector
product. Thus, we aim to reduce the number of required iterations.
Let $\vec{e}_k$ denote the true error at the $k$-th iteration (i.e.\ the difference between the computed
approximation at the $k$-th iteration and the true solution).
The well-known error estimate for the conjugate gradient method is given by
\cite{shewchuk1994introduction}:
\[
    \norm{\vec{e}_k}_{\Gamma} \leq 2 \lb \frac{\sqrt{\kappa(\Gamma)} - 1}{\sqrt{\kappa(\Gamma)} + 1} \rb^k \norm{\vec{e}_0}_{\Gamma},
\]
where $\kappa(\Gamma)$ is the spectral condition number of $\Gamma$, which is a function
of the contact areas and the ratio of the friction coefficients.

The error estimate suggests that a matrix with higher condition number
requires more iterations. However,
conjugate gradient convergence is more complex and often faster than this estimate suggests \cite{liesen2013krylov}.
While this estimate has limited practical use, it does motivate the reduction
of the condition number through preconditioning.

To precondition the system, we choose a symmetric matrix
$H = EE^T$ such that $\Gamma^{-1} \sim H$. We then solve the modified system
$E^T \Gamma E \hat{\vec{x}} = E^T \vec{F}$ and $\vec{x} = E\hat{\vec{x}}$.
This approach often reduces the iteration count if
$\kappa(E^T \Gamma E) < \kappa(\Gamma)$, and reduces run time if
we choose $H$ so that solving $H\vec{x} = \vec{b}$ is computationally inexpensive.

Preconditioning is crucial in many problems, particularly those
arising from the discretization of partial differential equations. While no unifying
theory exists, it is a well-developed field \cite{demmel1997applied,greenbaum1997iterative}.
Our focus on a matrix-free implementation limits the direct application of many
existing preconditioning techniques to our problem.

In off-lattice agent-based models, we can interpret the off-diagonal sparsity pattern
of $\Gamma$ as a graph $\mathfrak{C}$ (see Figure~\ref{fig:graph}). This graph representation
works as follows:

\begin{itemize}

    \item Each cell at position $\vec{r}_i$ is represented as the $i$-th vertex.

    \item We draw an edge $e = (i, j)$ between vertices $i$ and $j$ when the cells'
    contact area $A_{ij}$ is nonzero.

    \item The edges are weighted by the cell-cell friction matrices $\Gamma_{ij}^{cc}$.

    \item Cell-substrate matrices $\Gamma_i^{cs}$ are represented with weighted self-loops.

\end{itemize}

The result is an undirected, matrix-weighted, and labeled (by cell id) graph that represents
the friction matrix.

The relationship between matrices and graphs is well-established
\cite{scott2023algorithms} and underpins many algorithms for sparse matrices.
Typically, the matrix is constructed first, and its underlying
graph is derived subsequently. Our approach reverses this process: we start
with the collision graph constructed during the collision detection phase, and
derive the friction matrix from it.
Specifically, the friction matrix is the block Laplacian of the collision graph.

We employ a technique pioneered by Vaidya \cite{vaidya1991solving} that
uses subgraphs of $\mathfrak{C}$ as preconditioners. Subgraphs are advantageous
because they are sparse, yet capture much of the relevant information from $\mathfrak{C}$.
This characteristic allows them to effectively balance between approximating $\Gamma$
and computational efficiency. In our implementation, we specifically use a maximum spanning tree.
This tree can be constructed in linearithmic time, and its associated matrix can be factored
in linear time. The broader study of using subgraphs as preconditioners is
known as support graph theory.

Vaidya's original manuscript lacked many proofs, which were later provided in \cite{bern2006support}.
In our work, we extend these proofs to apply to block-structured matrices.
This extension allows us to obtain estimates for the smallest and largest eigenvalues of
such preconditioned systems. These eigenvalue bounds serve two important purposes: (1)
they provide worst-case convergence estimates, and (2) They are valuable in implementing the
robust conjugate gradient stopping criteria proposed by \cite{Axelsson2001,Meurant2024}.

\subsection{Outline}

The remainder of this paper is organized as follows:
In Section~\ref{sec:matrix_structure}, we introduce the linear system arising from
agent-based models and its natural graph structure. This section provides the foundation
for understanding the mathematical framework of our approach.
Section~\ref{sec:support_graph} focuses on Vaidya's preconditioners. We explain how to
construct these preconditioners for block-structured matrices and solve the resulting
systems in near-linear time. This section bridges the gap between graph theory
and numerical linear algebra.
Section~\ref{sec:convergence} presents our main theoretical contribution. Here, we extend
support graph theory to block-structured matrices and derive eigenvalue bounds for
the preconditioned linear system. This extension is crucial for applying
support graph theory to the matrices arising in agent-based models.
In Section~\ref{sec:numerics}, we present our numerical results. We demonstrate
the effectiveness of our preconditioner using a series of numerical benchmarks.
These experiments validate our theoretical findings and showcase the practical
benefits of our approach.
Finally, Section~\ref{sec:discussion} concludes the paper with a discussion of our findings,
and its implications for real-world use cases.
We situate our contributions in the broader theoretical landscape and discuss connections
to the existing literature on this problem. Additionally, we provide potential
directions for future research in this area.

\section{Preliminaries}\label{sec:matrix_structure}

This section introduces the basic graph and matrix structure of our problem. From the collision detection phase of an off-lattice simulation, we derive a matrix-vector equation whose solution represents the velocities of all the cells. We then show that the matrix we are solving is the block Laplacian of the collision graph, and we use this to prove positive definiteness. To set the stage for our subsequent discussion, we briefly outline the steps of an off-lattice model, considering a simulation of $n$ spherical cells.

\begin{enumerate}[label=\textbf{Step~\arabic*:}, ref=Step~\arabic*, leftmargin=*, labelindent=\parindent]

    \item \underline{Broad-phase collision detection.}
    Identify possible cell contact pairs $(i, j)$.
    Efficient divide-and-conquer algorithms, such as axis-aligned bounding boxes, are
    commonly used \cite{tracy2009efficient}.

    \item \underline{Compute forces between cells.}
    Examine the possible cell pairs identified in the previous step, and identify interacting
    cells. Then, compute their contact area and contact force using a physical model.
    We denote the contact area between cells $i$ and $j$ by $A_{ij}$.
    For spherical cells, Hertz or JKR contact mechanics are commonly
    used \cite{VANLIEDEKERKE2018}.

    \item \underline{Assemble the friction matrices.} Construct
    the $3 \times 3$ cell-cell and cell-substrate friction matrices. The cell-cell friction matrix between cells $i$ and $j$ is given by
    \begin{equation}\label{eq:cell_cell_friction}
        \Gamma_{ij}^{cc} = A_{ij}\lb \gamma_{\parallel} \vec{u}_{ij} \otimes \vec{u}_{ij}
            + \gamma_{\perp} \lb I - \vec{u}_{ij} \otimes \vec{u}_{ij} \rb\rb,
    \end{equation}
    where $\gamma_{\parallel}$ and $\gamma_{\perp}$ are the parallel and perpendicular coefficients of friction,
    respectively (both of which are positive), and $\vec{u}_{ij} \in \R^3$ is the unit contact vector
    \cite{VANLIEDEKERKE2018}. If $\vec{r}_i$ is the position of cell $i$, then
    \[
        \vec{u}_{ij} = \frac{\vec{r}_j - \vec{r}_i}{\norm{\vec{r}_j - \vec{r}_i}}.
    \]
    The cell-substrate friction matrix has several possible forms depending on the cell shape. In the isotropic case of spherical cells, we can write $\Gamma^{cs}_i = \lambda_{\mathrm{med}}I$, where $\lambda_{\mathrm{med}}$ is the coefficient of friction between the cell and the medium. For ellipsoidal cells, the cell-substrate friction matrix takes on a similar form to the cell-cell friction matrix in terms of directional friction coefficients and the unit direction vector. More complicated cell shapes may not be as easily expressible, but we only require that all the friction matrices are symmetric positive definite.

    \item \underline{Solve the equation of motion.} The equation of motion for cell $i$ is
    \begin{equation}\label{Eq:Motion}
        \Gamma^{cs}_i \vec{v}_i + \sum_{A_{ij} > 0} \Gamma^{cc}_{ij} (\vec{v}_i - \vec{v}_j) = \vec{F}_i,
    \end{equation}
    where $\vec{v}_{i}$ is the velocity vector and $\vec{F}_i$ is the
    total nonfrictional force acting on the cell.
    We interpret Equation~\eqref{Eq:Motion} as one row of a large linear system $\Gamma\vec{v}=\vec{F}$. For $n$ cells, the friction matrix $\Gamma$ is $3n \times 3n$, and
    we show that $\Gamma$ is symmetric
    positive definite. Hence, the conjugate gradient method is an efficient choice for obtaining
    an accurate solution.

    \item \underline{Update cell positions.}
    Frequently, a forward Euler method is used:
    \[
        \vec{r}_i(t + \Delta t) = \vec{r}_i(t) + \Delta t\, \vec{v}_i,
    \]
    where the step-size $\Delta t$ is chosen according to the Euler's method stability criterion.
    Higher-order integration methods are rarely used. Two exceptions are
    PhysiCell \cite{Ghaffarizadeh2018}, which employs a second-order Adam's-Bashford method with
    a fixed time-step, and \cite{MesenchymalCells}, where an embedded Runge-Kutta-23 method
    with an adaptive time-step is employed.

\end{enumerate}

Since the friction matrix $\Gamma$ is composed of $3\times 3$ blocks, we establish a few simple properties of these smaller friction matrices. Note that each cell-cell friction matrix (and the cell-substrate friction matrix of an ellipsoidal cell) is the sum of two orthogonal projectors. Let $\gamma_{\mathrm{max}}$ and $\gamma_{\mathrm{min}}$ be the maximum and minimum elements of $\{\gamma_\parallel, \gamma_\perp\}$, respectively.

\begin{lemma}\label{Lemma:SymPosDefBuildingBlocks}
    Let $\vec{u} \in \R^n$ be a unit vector, then the matrix
    \[
        \Upsilon = \gamma_{\parallel} \vec{u} \otimes \vec{u}
            + \gamma_{\perp} \lb I - \vec{u} \otimes \vec{u} \rb,
    \]
    \begin{enumerate}
        \item is symmetric positive definite;

        \item its eigenvalues are $\gamma_{\parallel}$ and $\gamma_{\perp}$ with multiplicities
            $1$ and $2$ respectively;

        \item its eigenspaces are $E_{\gamma_{\parallel}} = \spn\lb \vec{u} \rb$ and
            $E_{\gamma_{\perp}} = E_{\gamma_{\parallel}}^{\perp}$;

        \item its operator norm (with respect to the 2-norm) is $\norm{\Upsilon} = \gamma_{\mathrm{max}}$; and

        \item its condition number is
            \[
                \kappa(\Upsilon) =
                    \frac{\gamma_{\mathrm{max}}}{\gamma_{\mathrm{min}}}.
            \]
    \end{enumerate}
\end{lemma}

\begin{proof}
    The projection matrices $\vec{u} \otimes \vec{u}$ and $I - \vec{u} \otimes \vec{u}$ are
    clearly symmetric.
    Note that $(\vec{u} \otimes \vec{u})\vec{u} = \vec{u}$, so
    $\Upsilon \vec{u} = \gamma_{\parallel}\vec{u}$. Since the eigenvectors of a real symmetric
    matrix are orthogonal, take $\vec{w}$ such that $\vec{w}^T\vec{u} = 0$. Then $(\vec{u} \otimes \vec{u})\vec{w} = 0$,
    so $\Upsilon \vec{w} = \gamma_{\perp} \vec{w}$. The vector $\vec{w}$ lies in the orthogonal
    complement of $\vec{u}$, which is two-dimensional.
    Both of the eigenvalues are positive, which gives positive definiteness. The operator norm of a symmetric positive definite matrix is the maximum eigenvalue.
\end{proof}

Ill-conditioned cell-cell friction matrices will make our preconditioners less effective even if $\Gamma$ as a whole is well-conditioned. In the extreme case of rank-deficient friction matrices, perhaps representing freely rotating objects, our condition number bounds in Section \ref{sec:convergence} do not hold and support graph preconditioners are likely a poor choice.

Observe that the off-diagonal sparsity pattern of $\Gamma$ is determined
by the interacting cell pairs. We interpret this as a matrix-weighted graph.

\begin{definition}[Matrix-weighted graph]
    A matrix-weighted graph is a triple $G = (V, E, w)$ where $(V, E)$ is an undirected graph with a matrix-valued weight function $w\colon V \times V \to \R^{d \times d}$. We require that $w(e)$ is positive definite for all $e \in E$, that $w(v, v)$ is positive semidefinite for all $v \in V$, and that $w(u, v) = 0$ otherwise.
\end{definition}

We allow nonzero weights on pairs $(v, v)$ for convenience when representing cell-substrate friction. To refer to the number of edges $|E|$ and the number of vertices $|V|$, we use $m$ and $n$ respectively. We reserve the Fraktur font for objects relating to the collision graph. Let $\mathfrak{D}$ be the set of cells, $\mathfrak{E}$ the set of interacting cell pairs, and $\mathfrak{w}: \mathfrak{D} \times \mathfrak{D} \to \R^{3 \times 3}$ the weight function with $\mathfrak{w}(i, j) = \Gamma^{cc}_{ij}$ for all $(i, j) \in \mathfrak{E}$ and $\mathfrak{w}(i, i) = \Gamma^{cs}_i$ for all $i \in \mathfrak{D}$. These form the collision graph $\mathfrak{C} = (\mathfrak{D}, \mathfrak{E}, \mathfrak{w})$. This graph is typically very sparse with $m = \mathcal{O}(n)$.

Most of the matrices in this paper have a block structure. If $A$ is an $m \times n$ block matrix, then we mean that $A$ is $m$ blocks tall by $n$ blocks wide. The precise size of the blocks is not important, but we assume they are square. To be explicit when indexing block matrices, we use an underline to index over blocks. For example, define a $2 \times 2$ block matrix
\[
    A = \begin{pmatrix}
        W & X \\ Y & Z
    \end{pmatrix}.
\]

Then $A_{11} = W_{11}$ and $\mat{A_{11}} = W$.
We also use this notation for vectors. If $\vec{a} = \begin{pmatrix}
    \vec{x} \\ \vec{y}
\end{pmatrix}$, then $\vec{a}_1 = \vec{x}_1$ and $\mat{\vec{a}_1} = \vec{x}$.

To formalize the relation between the friction matrix $\Gamma$ and the collision graph $\mathfrak{C}$, we introduce the block Laplacian.

\begin{definition}[Block Laplacian]
    A block Laplacian is a symmetric block matrix whose off-diagonal blocks are either zero or negative definite, and whose block row and column sums are positive semidefinite. For a matrix-weighted graph $G = (V, E, w)$, the block Laplacian of $G$ is the matrix $L$ where
    \[
        \mat{L_{ij}} = \begin{cases}
            -w(i, j) & i \neq j, \\
            w(i, i) + \sum_{k \neq i} w(i, k) & i = j.
        \end{cases}
    \]
\end{definition}

Note that block Laplacians are not diagonally dominant in general, but the condition we impose on their block row and column sums is analogous. We show in Section \ref{sec:convergence} that this generalized notion of block diagonal dominance is sufficient to apply support graph preconditioners.

Another mathematical structure that seems closely related to the block Laplacian at first glance is the connection Laplacian introduced by Singer and Wu in \cite{singerwu}. However, a key difference between the two structures is that connection Laplacians with real entries require the edge weights to be orthogonal, whereas block Laplacians only require positive definiteness. It is almost never the case that the collision graph yields a connection Laplacian because this would require both friction coefficients to be 1. The solver presented in \cite{Kyng2015SparsifiedCA} works on a class matrices satisfying a certain definition of block diagonal dominance which connection Laplacians belong to but block Laplacians generally do not. So the theory of connection Laplacians is not generally applicable to block Laplacians.

It is easy to see that $\Gamma$ is the block Laplacian of $\mathfrak{C}$. The following lemma establishes the definiteness needed to apply the conjugate gradient method.

\begin{lemma}
\label{blapposdef}
    All block Laplacians are positive semidefinite, and they are positive definite when their block row sums are positive definite.
\end{lemma}

\begin{proof}
    Let $A$ be a block Laplacian. For all $\vec{x}$,
    \begin{align*}
        \vec{x}^TA\vec{x} &= \sum_i \mat{\vec{x}_i^TA_{ii}\vec{x}_i} + \sum_{i \neq j} \mat{\vec{x}_i^TA_{ij}\vec{x}_j} \\
        &\geq \sum_{i \neq j} |\mat{\vec{x}_i^TA_{ij}\vec{x}_i}| + \sum_{i \neq j} \mat{\vec{x}_i^TA_{ij}\vec{x}_j} \\
        &= \sum_{i > j} \left(|\mat{\vec{x}_i^TA_{ij}\vec{x}_i}| + |\mat{\vec{x}_j^TA_{ij}\vec{x}_j}| + 2(\mat{\vec{x}_i^TA_{ij}\vec{x}_j})\right).
    \end{align*}
    This inequality is strict when the block row sums are positive definite. By the AM-GM and Cauchy-Schwarz inequalities respectively,
    \[
        |\mat{\vec{x}_i^TA_{ij}\vec{x}_i}| + |\mat{\vec{x}_j^TA_{ij}\vec{x}_j}| \geq 2\sqrt{|\mat{\vec{x}_i^TA_{ij}\vec{x}_i}| \cdot |\mat{\vec{x}_j^TA_{ij}\vec{x}_j}|} \geq 2|\mat{\vec{x}_i^TA_{ij}\vec{x}_j}|.
    \]
    This implies that each term in the summation above is nonnegative, which yields the desired result.
\end{proof}

\section{Support Graph Preconditioners}\label{sec:support_graph}

To effectively precondition $\Gamma$, we need to find an easily factorable matrix that closely approximates it. This section introduces support graph theory and Vaidya's preconditioners as tools to do this, along with efficient graph algorithms for their implementation. Typically, an underlying graph is derived from a given matrix. This graph is manipulated (e.g., by taking a subgraph) and its Laplacian is used as a preconditioner. However, since the friction matrix is derived from the collision detection phase, it is more appropriate for us to view $\Gamma$ as the underlying matrix of $\mathfrak{C}$, and we can derive preconditioners from manipulations (e.g., subgraphs) of the collision graph. This is why support graph preconditioners are a natural choice for simulations. In fact, we can entirely avoid assembling matrices by working with the collision graph. All we need are the contact areas, normal vectors, and friction coefficients.

\subsection{Vaidya's Preconditioners}
\label{sec:vaidya}

The first of Vaidya's preconditioners is the maximum spanning tree (MST) preconditioner. The idea is to precondition the Laplacian of a graph with the Laplacian of an MST. We work with trees because their Laplacians can be factored in linear time, and MSTs in particular because they capture a lot relevant information about the graph. In other words, an MST ``supports'' its parent graph well.

Since we are working with matrix-weighted graphs, we define an MST with respect to the minimum eigenvalues of the weights, a choice that is justified in the next section. In the case of $\mathfrak{C}$, this is equivalent to weighting by contact area. We also include the self-loops (i.e., the cell-substrate friction) in the MST. Let $\mathfrak{T}$ be an MST of $\mathfrak{C}$ and let $P$ be its block Laplacian. We call $\mathfrak{T}$ a support graph of $\mathfrak{C}$ and $P$ an MST preconditioner of $\Gamma$.

Vaidya's second class of preconditioners builds on the first by adding edges back to an MST. Given a parameter $t$, we split an MST into $t$ disjoint subtrees of roughly the same size where each subtree has at most $m/t$ vertices. Then we add the maximum weight edge in $\mathfrak{C}$ between each pair of subtrees if they are connected in $\mathfrak{C}$. Let $\mathfrak{T}'$ be an augmented MST and $P'$ its block Laplacian. We call $P'$ an augmented MST preconditioner of $\Gamma$. The theoretically optimal value of $t$ is approximately $n^{1/4}$ \cite{chen2003vaidya}.

The only step left to define is how we generate the support graph. Prim's algorithm is the best choice for finding an MST because it tells us, at no extra cost, how to permute the rows and columns of the block Laplacian to generate zero fill during factorization. We prove this in the next subsection. The time complexity of Prim's algorithm is $\mathcal{O}(m\log n)$. Augmenting an MST is straightforward once it has been decomposed into subtrees. A simple partitioning algorithm is presented in \cite[\textsc{TreePartition}]{chen2003vaidya} and a more sophisticated one in \cite{spielman2014nearly}.

\subsection{The Elimination Game}
\label{subsec:elimination}

When factoring or performing Gaussian elimination on a matrix, new nonzero entries may be created, changing the sparsity pattern of the matrix. These new entries, called fill, require more memory and slow down computations. However, permuting the rows and columns of the matrix can change the amount of fill. There is a graphical interpretation of Gaussian elimination that shows how fill is created, called the elimination game \cite{pothen2018elimination}.

In the game, all the vertices of a graph are eliminated in some order (e.g., see Figure~\ref{fig:elimination_game} where the elimination order is according to the vertex labels). When a vertex $v$ is eliminated, fill edges are constructed so that the uneliminated neighbors of $v$ (connected by either an original or fill edge) become pairwise adjacent. In other words, the uneliminated neighbors of $v$ become a clique. The set of fill edges is in bijection with the set of fill entries that would be created during Gaussian elimination or decomposition. Note that the elimination game is never explicitly implemented but rather implicitly performed. Using this game, we can easily analyze the amount of fill created by a given permutation. The general problem of minimizing fill is NP-hard \cite{Yannakakis1981}, but trees can easily be ordered to produce no fill.

\begin{figure}[!ht]\centering
    \includegraphics[width=0.9\textwidth]{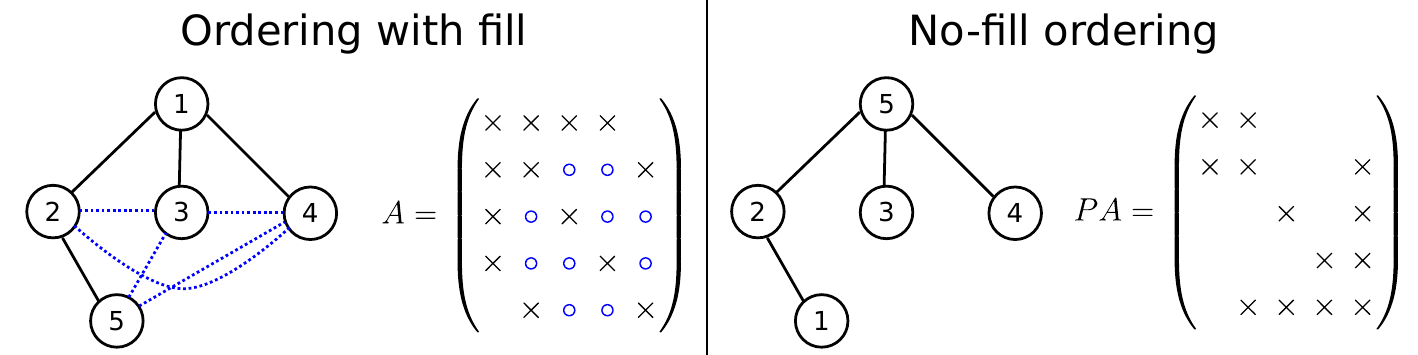}
    \caption{An example of the elimination game being played on the same graph with different
    orderings. Graph edges and matrix entries are in solid black and denoted by crosses
    respectively. Fill edges and entries are in dotted blue and blue circles respectively.
    }\label{fig:elimination_game}
\end{figure}

\begin{lemma}
    \label{nofill}
    An ordering of the vertices of a rooted tree where no vertex occurs before any of its children produces no fill in the elimination game.
\end{lemma}

\begin{proof}
    The first vertex eliminated must be a leaf. This produces no fill and yields another rooted tree. Inductively, the hypothesis demands that each vertex must be a leaf when it is eliminated, otherwise, it would have children to eliminate first. Since eliminating leaves produces no fill, we get the desired result.
\end{proof}

\begin{cor}
    The reverse order in which vertices are added to the MST in Prim's algorithm produces no fill in the elimination game.
\end{cor}

For augmented trees, reducing fill is far less simple. Reversing the traversal order of Prim's algorithm can yield very poor results. For example, if the added edges are between vertices that occur late in the traversal order, then lots of unnecessary fill is created. Of course, vertices that are not the ancestor of (we say a vertex is its own ancestor) an endpoint of an added edge can be eliminated as before without producing any fill. The rest of the vertices can be eliminated in the order specified by a fill-reducing algorithm. Both GENMMD and METIS are tested in \cite{chen2003vaidya}, but it is worth mentioning that Chen and Toledo only work with matrices whose underlying graphs are regular meshes. Cell-based models tend to produce more irregularities in their graph structure. A variety of algorithms like reverse Cuthill-McKee may also perform well \cite{duff1989}.

\subsection{Solving Preconditioners}

Matrix inverses are seldom computed explicitly. Instead, large matrices are decomposed into the product of diagonal and triangular matrices in which solving systems is easy. For a symmetric positive definite matrix $M$, the Cholesky decomposition finds a lower triangular matrix $L$ such that $M = LL^T$. We instead opt for the similar $LDL^T$ decomposition in which $D$ is a diagonal matrix and $L$ has only ones on its diagonal. This is preferable for block matrices because it avoids computing matrix square roots. This section provides a general decomposition algorithm for symmetric positive definite block matrices and adapts it to the special case of nonsingular block Laplacians of trees.

\begin{algorithm}
    \caption{\\It is worth restating this algorithm because we cannot take commutativity for granted as is often done.}
    \label{alg:ldlt}
    \begin{algorithmic}[1]
    \Function{LDLT}{$A$: $n \times n$ symmetric positive definite block matrix}
    \State $L, D \gets$ $n \times n$ block matrices
    \For{$i \gets 1, n$} \qquad$\triangleright$ Current column
        \State $X \gets 0$
        \For{$j \gets 1, \ldots, i-1$}
            \State $X \gets X + \mat{L_{ij}D_{jj}L_{ij}}$
        \EndFor
        \State $\mat{D_{ii}} \gets \mat{A_{ii}} - X$
        \State $\mat{L_{ii}} \gets I$
        \For{$j \gets i+1, \ldots, n$} \qquad$\triangleright$ Current row
            \State $Y \gets 0$
            \For{$k \gets 1, \ldots, i-1$}
                \State $Y \gets Y + \mat{L_{ik}D_{kk}L_{jk}}$ \qquad$\triangleright$ Subtract previous outer products
            \EndFor
            \State $\mat{L_{ji}} \gets (\mat{A_{ji}} - Y)\mat{D^{-1}_{ii}}$
        \EndFor
    \EndFor
    \State\Return $(L, D)$
    \EndFunction
    \end{algorithmic}
\end{algorithm}

The general $LDL^T$ algorithm takes $\mathcal{O}(n^3)$ time, but we show that the preconditioners we want to factor are sparse and only take linear or near-linear time. Another perk of this algorithm is that $L$ can be calculated in-place because previous entries in $A$ are not reused.

\begin{algorithm}
    \caption{}
    \label{alg:treeldlt}
    \begin{algorithmic}[1]
    \Function{TreeLDLT}{$A$: nonsingular block Laplacian of a rooted tree}
    \State $L, D \gets$ $n \times n$ block matrices
    \ForAll{vertices $i$ in decreasing order of distance to the root}
        \State $X \gets 0$
        \ForAll{children $j$ of $i$}
            \State $X \gets X + \mat{L_{ij}D_{jj}L_{ij}}$
        \EndFor
        \State $\mat{D_{ii}} \gets \mat{A_{ii}} - X$
        \State $\mat{L_{ii}} \gets 1$
        \State $j \gets$ parent of $i$
        \State $\mat{L_{ji}} \gets \mat{A_{ji}D^{-1}_{ii}}$
    \EndFor
    \State\Return $(L, D)$
    \EndFunction
    \end{algorithmic}
\end{algorithm}

\begin{theorem}
    Algorithm \ref{alg:treeldlt} returns the $LDL^T$ factorization of the nonsingular block Laplacian of a matrix-weighted tree in $\mathcal{O}(n)$ time. The sparsity pattern of $L$ is the same as the lower triangle of the original matrix.
\end{theorem}

\begin{proof}
    We assume the validity of Algorithm \ref{alg:ldlt} and prove its equivalence to Algorithm \ref{alg:treeldlt} in this special case. Say that the rows and columns of $A$ are ordered as specified on line 3 of Algorithm \ref{alg:treeldlt}, so $i < j$ means that vertex $i$ is at least as far from the root as vertex $j$. In this case, we say that $i$ is younger than $j$ and that $j$ is older than $i$. Note that vertices $i$ and $j$ being connected is equivalent to $\mat{A_{ij}}$ being nonzero. Lines 5--7 in both algorithms behave identically because the only vertices younger than $i$ that are also connected to $i$ are its children.

    We want to show the equivalence of lines 10--16 in Algorithm \ref{alg:ldlt} to lines 10--11 in Algorithm \ref{alg:treeldlt}. We prove by induction on $i$ that, in Algorithm \ref{alg:ldlt}, $\mat{L_{ji}}$ is nonzero only if $i$ is connected to $j$ for all $j > i$. The hypothesis holds for the first vertex because $Y$ is guaranteed to be zero as there are no younger vertices. Now assume that the hypothesis holds for all vertices younger than some $i$. If $Y$ is nonzero, then there exists an older vertex $j$ and a younger vertex $k$ such that $k$ is connected to both $i$ and $j$. However, the only vertex older than $k$ that it is connected to is its unique parent. This implies that $i = j$, which contradicts line 10. Therefore, $Y$ is zero and the hypothesis is true, meaning the only $j$ we need to consider is the parent of $i$.
\end{proof}

When there is zero fill, no additional memory allocation is required for $L$ even when using a matrix-free or sparse matrix implementation. However, this is not the case for augmented MST preconditioners. Algorithm \ref{alg:treeldlt} is no longer valid and extra fill entries need to be stored for a complete factorization. Other options for solving systems in the augmented MST preconditioner include using an incomplete factorization with a fill-reducing algorithm like one of those mentioned in Section \ref{subsec:elimination}, or making a nested call to the conjugate gradient method with the standard MST as the preconditioner.

This begs the question of whether there exists an effective augmentation strategy with which complete factorizations do not create fill. This restricts us to only add back edges between siblings in the MST. However, we suspect that this would not yield promising results for reasons we formalize in the next section. Intuitively, we want to add back edges that drastically reduce the distance between pairs of vertices (i.e., reduce the stretch), but sibling edges do a poor job of this.

Another benefit of the standard MST preconditioner is that solving systems in the decomposed block Laplacian takes linear time. We provide algorithms for this in Section \ref{sec:treesolve}. Finally, it is known that block $LU$ factorization for symmetric positive definite
matrices is stable as long as the matrix is well-conditioned \cite{demmel1995stability,Higham2002}.

\section{Block-Structured Support Graph Theory}\label{sec:convergence}

The existing literature on support graphs focuses entirely on symmetric diagonally dominant matrices. Here, we generalize a sequence of lemmas from \cite{bern2006support} and \cite{Gremban1996} to work with block Laplacians as well, and we refer the reader to Appendix~\ref{app:posdefoff} for further generalizations to matrices with positive definite off-diagonal blocks. We provide proofs where they differ from the non-block case. The goal is to show that both of Vaidya's preconditioners achieve a minimum eigenvalue of at least 1, and that the condition number is $\mathcal{O}(\kappa mn)$ with an MST preconditioner and $\mathcal{O}(\kappa n^2/t^2)$ with an augmented MST preconditioner, where $\kappa$ is the maximum condition number of all the edge weights and we assume sufficient sparsity.

We use $\succeq$ to represent the Loewner order. For matrices $A$ and $B$, we say $A \succeq B$ if and only if $A - B$ is positive semidefinite. A fact that we use without proof is that the Loewner order is a partial order on symmetric matrices. We denote the set of finite generalized eigenvalues of a pair of matrices $(A, B)$ by $\lambda(A, B)$. That is to say, $\lambda(A, B)$ is the set of numbers $\lambda$ such that there exists a vector $\vec{x}$ where $A\vec{x} = \lambda B\vec{x}$. If $B$ is a preconditioner for $A$, then the condition number of the preconditioned system $B^{-1}A$ is precisely the ratio of the extremal finite generalized eigenvalues $\lambda_{\mathrm{max}}(A, B)/\lambda_{\mathrm{min}}(A, B)$.

We start by defining the support of a pair of matrices, which bounds the maximum eigenvalue of the pair.
\begin{definition}
    The support of a pair of matrices $(A, B)$ is
    \[
        \sigma(A, B) = \min\{\tau \mid \tau B \succeq A\}.
    \]
    If no such $\tau$ exists, then we say $\sigma(A, B) = \infty$.
\end{definition}

\begin{lemma}
    Suppose $A$ and $B$ are positive semidefinite matrices. Then
    \[
        \lambda_{\mathrm{max}}(A, B) \leq \sigma(A, B),
    \]
    and equality holds when the support is finite.
\end{lemma}

\begin{proof}
    See \cite[Lemma 4.4]{Gremban1996}
\end{proof}

Since $\lambda_{\mathrm{max}}(B, A) = \lambda_{\mathrm{min}}(A, B)^{-1}$, the supports $\sigma(A, B)$ and $\sigma(B, A)$ are all we need to bound the condition number $\kappa(B^{-1}A)$. In fact, we already have the necessary tools to bound the minimum eigenvalue.

\begin{theorem}
    Let $G = (V, E, w)$ be a matrix-weighted graph with subgraph $H = (V, F, w')$ and block Laplacians $L_G$ and $L_H$. Then $\lambda_{\mathrm{min}}(L_G, L_H) \geq 1$.
\end{theorem}

\begin{proof}
    Observe that $L_G = L_H + L_K$ where $L_K$ is the block Laplacian of the graph $K = (V, E \setminus F, w - w')$. Since $L_G - L_H = L_K$ is a block Laplacian, it is positive semidefinite. This implies that $\sigma(L_H, L_G) \leq 1$ and that $\lambda_{\mathrm{min}}(L_G, L_H) \geq 1$.
\end{proof}

Bounding the maximum eigenvalue requires more effort. To simplify computing the support of a matrix and preconditioner, we use the following lemma to decompose the matrices into sums of positive semidefinite matrices.

\begin{lemma}
    \label{maxsupp}
    Let $A = A_1 + \cdots + A_k$ and $B = B_1 + \cdots + B_k$ where each $A_i$ and $B_i$ is positive semidefinite. Then
    \[
        \sigma(A, B) \leq \max_i\{\sigma(A_i, B_i)\}.
    \]
\end{lemma}

\begin{proof}
    See \cite[Lemma 4.7]{Gremban1996}
\end{proof}

A further simplification we can make is to only consider block Laplacians with zero block row sums. This means that we can ignore cell-substrate friction in the rest of the analysis using the following lemma.

\begin{lemma}
    Let $A$ be a block Laplacian and define $A'$ to be the matrix with the same off-diagonal blocks as $A$ and zero block row sums. Let $B'$ be a block matrix and $B = B' + A - A'$. If $\beta B' \succeq A'$ for some $\beta \geq 1$, then $\beta B \succeq A$. Similarly, if $\alpha A' \succeq B'$ for some $\alpha \geq 1$, then $\alpha A \succeq B$.
\end{lemma}

\begin{proof}
    See \cite[Lemma 2.5]{bern2006support}
\end{proof}

The lower bound on $\alpha$ and $\beta$ is not important because the pairs of matrices with which we are concerned have supports of at least 1. Next, we prove a small lemma that is helpful in the rest of the section.

\begin{lemma}
    \label{posdefloewner}
    Let $A$ be a symmetric positive semidefinite matrix. Then
    \[
        \lambda_{\mathrm{max}}(A)I \succeq A \succeq \lambda_{\mathrm{min}}(A)I.
    \]
\end{lemma}

\begin{proof}
    Since $A$ is symmetric positive semidefinite, it is diagonalizable and we can write $A = PDP^{-1}$ where $D$ is the diagonal matrix whose entries are the eigenvalues of $A$. For concision, let $\lambda = \lambda_{\mathrm{max}}(A)$. Observe
    \[
        \lambda I - A = \lambda I - PDP^{-1} = P(\lambda I - D)P^{-1}.
    \]
    By the definition of $\lambda$, the middle factor of $\lambda I - D$ has all nonnegative entries, so the whole difference is positive semidefinite. Similar logic applies for the minimum eigenvalue.
\end{proof}

Now we are ready to prove the main result. We decompose a matrix and support graph preconditioner into sums of the block Laplacians of individual edges and paths. Then we analyze their pairwise supports with the following three lemmas.

\begin{lemma}
    \label{dilation_same_blocks}
    Suppose $A$ and $B$ are symmetric positive definite matrices. Let
    \[
        \hat{A} = \begin{pmatrix}
            A & 0 & \cdots & 0 & -A \\
            0 & 0 & & & 0 \\
            \vdots & & \ddots & & \vdots \\
            0 & & & 0 & 0 \\
            -A & 0 & \cdots & 0 & A
        \end{pmatrix},
    \quad\mbox{and}\quad
        \hat{B} = \begin{pmatrix}
            A & -A & & & \\
            -A & 2A & -A & & \\
            & & \ddots & & \\
            & & -A & 2A & -A & \\
            & & & -A & A
        \end{pmatrix},
    \]
    be $(k+1) \times (k+1)$ block matrices. Then $k\hat{B} \succeq \hat{A}$.
\end{lemma}

\begin{proof}
    Let $C = k\hat{B} - A$ and define $\hat{C} = \mathrm{diag}(A^{-1}) C$ which is
    a block-structured matrix in which the blocks are either nonzero or a multiple
    of the identity matrix. Since $\mathrm{diag}(A^{-1})$ is positive definite, it follows
    that $C$ is semi-positive definite if $\hat{C}$ is positive semidefinite.
    We prove that $\hat{C}$ is positive semidefinite
    using induction as in the non-block-structured
    case in \cite[Lemma 2.7]{bern2006support}. The only difference is that we use a
    block-structured symmetric Gaussian elimination, meaning the $i$-th
    elimination step is
    \[
        \hat{C}_i = E_{i} \hat{C}_{i - 1} E_{i}^T
    \]
    where $E_i$ is block-structured with identity blocks along its diagonal and
    two nonzero off-diagonal blocks:
    \[
        \mat{E_{1i}} = \frac{1}{1 + i} I,\ \mat{E_{(i+1)(i)}} = \frac{1}{i + 1} I.
    \]
    At completion of this process we obtain the matrix
    \[
        \hat{C} = \mathrm{diag}\lb 0, 2kI, \frac{3k}{2}I, \dots, \lb\frac{i+1}{i}\rb kI, \dots, 0 \rb.
    \]
    Since the matrix $\hat{C}$ has nonnegative values on its diagonal, this shows that
    $\hat{C}$ is positive semidefinite.
\end{proof}

\begin{lemma}
    \label{precongdil}
    Suppose $A$ and $B$ are symmetric positive definite matrices. Let
    \[
        \hat{A} = \begin{pmatrix}
            A & 0 & \cdots & 0 & -A \\
            0 & 0 & & & 0 \\
            \vdots & & \ddots & & \vdots \\
            0 & & & 0 & 0 \\
            -A & 0 & \cdots & 0 & A
        \end{pmatrix},
    \quad\mbox{and}\quad
        \hat{B} = \begin{pmatrix}
            B & -B & & & \\
            -B & 2B & -B & & \\
            & & \ddots & & \\
            & & -B & 2B & -B & \\
            & & & -B & B
        \end{pmatrix},
    \]
    be $(k+1) \times (k+1)$ block matrices. Then $k\cdot\lambda_{\mathrm{max}}(AB^{-1})\hat{B} \succeq \hat{A}$.
\end{lemma}

\begin{proof}
    Let $\lambda = \lambda_{\mathrm{max}}(AB^{-1})$. When $A = B$,
    we have that $\lambda = 1$ and the statement is equivalent to Lemma~\ref{dilation_same_blocks}.
    When $A \neq B$, we can reduce to the case of equality by multiplying $\hat{B}$ by $\mathrm{diag}(AB^{-1})$. This yields
    \[
        k\cdot\mathrm{diag}(AB^{-1})\hat{B} \succeq \hat{A}.
    \]
    Next we show that $k\cdot\lambda\hat{B} \succeq k\cdot\mathrm{diag}(AB^{-1})\hat{B}$.
    This is equivalent to showing that $\lambda I \succeq AB^{-1}$, which is true by Lemma~\ref{posdefloewner}.
    The rest follows from the transitivity of the Loewner order.
\end{proof}

Note that, up to permutation, the block Laplacian of a single edge looks like $\hat{A}$ and that of a simple path with uniform edge weights looks like $\hat{B}$. The following congestion-dilation lemma generalizes the previous one so that $\hat{B}$ can have varied edge weights.

\begin{lemma}[congestion-dilation lemma]
    \label{congdil}
    Let
    \[
        \hat{A} = \begin{pmatrix}
            A & 0 & \cdots & 0 & -A \\
            0 & 0 & & & 0 \\
            \vdots & & \ddots & & \vdots \\
            0 & & & 0 & 0 \\
            -A & 0 & \cdots & 0 & A
        \end{pmatrix}
    \]
    and
    \[
        \hat{B} = \begin{pmatrix}
            C_1 & -B_1 & & & \\
            -B_1 & C_2 & -B_2 & & \\
            & & \ddots & & \\
            & & -B_{k-1} & C_k & -B_k & \\
            & & & -B_k & C_{k+1}
        \end{pmatrix}
    \]
    be $(k+1) \times (k+1)$ block matrices where $A$, $B_i$, and $C_i$ are symmetric positive definite for all $i$ and $\hat{B}$ has zero block row sums. Then
    \[
        k \cdot \max_i\{\lambda_{\mathrm{max}}(AB_i^{-1})\} \cdot \hat{B} \succeq \hat{A}.
    \]
\end{lemma}

\begin{proof}
    Let $D = \min_i\{\lambda_{\mathrm{min}}(B_i)\} \cdot I$. Decompose
    \begin{align*}
        \hat{B} = \hat{B}_1 + \hat{B}_2 = &\begin{pmatrix}
            D & -D & & & \\
            -D & 2D & -D & & \\
            & & \ddots & & \\
            & & -D & 2D & -D & \\
            & & & -D & D
        \end{pmatrix} \\
        &+ \begin{pmatrix}
            C_1 - D & -B_1 + D & & & \\
            -B_1 + D & C_2 - 2D & -B_2 + D & & \\
            & & \ddots & & \\
            & & -B_{k-1} + D & C_k - 2D & -B_k + D \\
            & & & -B_k + D & C_{k+1} - D
        \end{pmatrix}.
    \end{align*}
    Let $\lambda = \max_i\{\lambda_{\mathrm{max}}(AB_i^{-1})\}$ and write
    \[
        k \cdot \lambda \hat{B} - \hat{A} = (k \cdot \lambda\hat{B}_1 - \hat{A}) + (k \cdot \lambda \hat{B}_2).
    \]
    The first summand is positive semidefinite by Lemma \ref{precongdil}. The diagonal blocks of $\hat{B_2}$ are positive semidefinite and the nonzero off-diagonal blocks are negative semidefinite by Lemma~\ref{posdefloewner}. This and the fact that the block row sums are zero mean the second summand is a block Laplacian, so it is positive semidefinite.
\end{proof}

Under the support of the path represented by $\hat{B}$, we call $\lambda$ the \textit{congestion} of the edge represented by $\hat{A}$, and $k$ is its \textit{dilation}. A more concise statement of the lemma is that $\sigma(\hat{A}, \hat{B})$ is bounded above by the product of the congestion and dilation. We use the this lemma to prove an upper bound on the maximum eigenvalue of a block Laplacian with an MST preconditioner. The following proofs make use of the fact that
\[
    \frac{\lambda_{\mathrm{max}}(A)}{\min_i\{\lambda_{\mathrm{min}}(B_i)\}} \geq \max_i\{\lambda_{\mathrm{max}}(AB_i^{-1})\}.
\]

\begin{theorem}
    Let $G = (V, E, w)$ be a matrix-weighted graph and $T$ its MST weighted by minimum eigenvalues. Let $L_G$ and $L_T$ be their block Laplacians and let $\kappa$ be the maximum condition number of all the edge weights in $G$. Then
    \[
        \lambda_{\mathrm{max}}(L_G, L_T) \leq \kappa m(n-1).
    \]
\end{theorem}

\begin{proof}
    For every edge $e = (u, v)$, let $p(e)$ be the path in $T$ from $u$ to $v$ that uses at least $1/m$ fraction of each edge weight. We can write
    \[
        L_G = \sum_{e \in E} L_e \quad \text{and} \quad L_T = \sum_{e \in E} L_{p(e)}
    \]
    where $L_e$ and $L_{p(e)}$ are the block Laplacians of the edges and paths respectively. By Lemma \ref{maxsupp},
    \[
        \sigma(L_G, L_T) \leq \max_{e \in E}\{\sigma(L_e, L_{p(e)})\}.
    \]
    The maximum possible length (edge count) of each $p(e)$ is $n-1$. This is the dilation. Since $T$ is an MST, the minimum eigenvalue of each edge weight in $p(e)$ is at least that of $w(e)$. This means that the congestion is at most $\kappa m$. The congestion-dilation lemma yields the desired result.
\end{proof}

Augmented MST preconditioners have a better upper bound that can be proven similarly.

\begin{theorem}
    Let $G = (V, E, w)$ be a matrix-weighted graph with augmented MST $T'$ as described in Section \ref{sec:vaidya}. Let $L_G$ and $L_{T'}$ be their block Laplacians and let $\kappa$ be the maximum condition number of all the edge weights in $G$. If every vertex has at most $d$ neighbors, then
    \[
        \lambda_{\mathrm{max}}(L_G, L_{T'}) \leq \frac{2\kappa d^3n^2}{t^2}.
    \]
\end{theorem}

\begin{proof}
    We perform a similar decomposition to the previous proof. Let $S_1, \ldots, S_t$ be the subtrees of $T'$. For each edge $e = (u, v)$, a path from $u$ to $v$ in $T'$ is not necessarily unique. If both $u$ and $v$ are in the same $S_i$, let $p(e)$ be the unique path from $u$ to $v$ contained in $S_i$. If $u$ is in $S_i$ and $v$ is in $S_j$ with $i \neq j$, let $p(e)$ be the concatenation of the following paths: the unique path in $S_i$ from $u$ to the endpoint of the edge that connects $S_i$ and $S_j$, that edge itself, and the unique path in $S_j$ from the other endpoint to $v$.

    Now we must decide what fraction of each edge weight to use. Each edge $e$ in $S_i$ can be in a support path from any of the $dn/t$ vertices in $S_i$ to any of their $d$ neighbors. This is $d^2n/t$ total paths, so we use $t/(d^2n)$ fraction of each edge weight. Following the same logic as in the previous proof, we get that the congestion is at most $\kappa(d^2n/t)$.

    In the worst case, one of these paths may go across all $dn/t - 1$ edges in one subtree, the edge connecting it to another subtree, and all $dn/t - 1$ edges in the other subtree. Therefore, the dilation is less than $2(dn/t)$ and the rest follows from the congestion-dilation lemma.
\end{proof}

In the context of collision graphs, $\kappa$ is simply $\max_{i, j}\{A_{ij}\}\cdot \gamma_{\mathrm{max}}/\gamma_{\mathrm{min}}$. Since the minimum eigenvalue is at least 1, the preceding bounds on the maximum eigenvalue are also bounds on the condition number of the preconditioned system.

\section{Numerical Benchmarks}
\label{sec:numerics}

\begin{figure}[htpb]
    \includegraphics[width=0.9\textwidth]{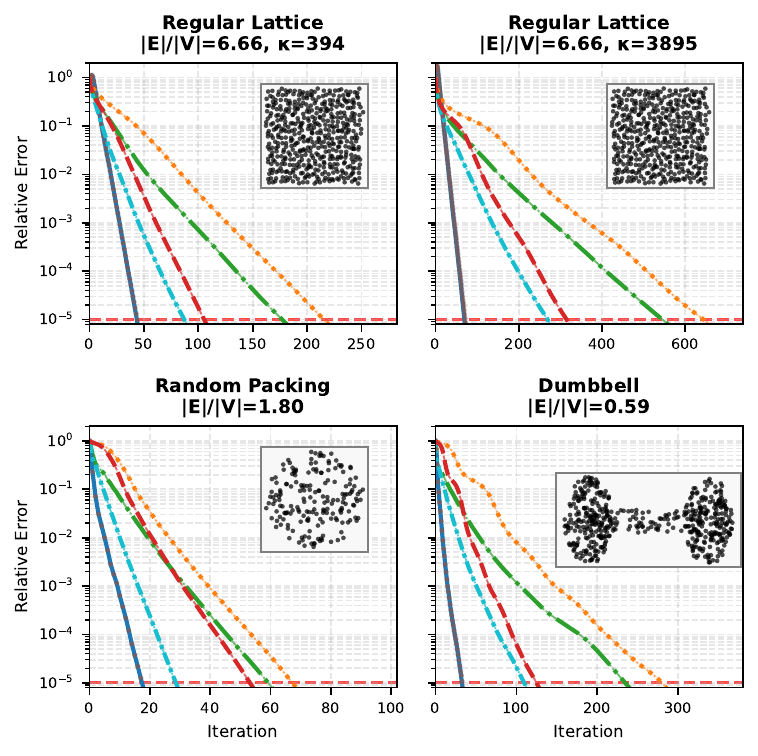}
    \caption{
    Preconditioner convergence behavior for different cell arrangements:
    hexagonal lattices with varying disorder (top row), random packing (bottom left),
    and bridged spheroids (bottom right).
    The insets show cross-sections through the 3D cell configurations.
    The red dashed line in each plot indicates
    the solver relative tolerance $10^{-5}$.
    {Legend:} Identity/no preconditioner (orange dotted line); block Jacobi
    (green long dash-dot); diagonally-shifted block IC(0) (dashed red); block Gauss-Seidel
    (cyan dash-dot); MST preconditioner (solid blue); augmented MST (gray circular markers).
    {Top Row:} Cells arranged on a three-dimensional hexagonal lattice with positional
    noise of mean zero and standard deviation $0.3r$ applied to each cell.
    Each configuration contains $n \approx 50\,000$ cells. The edge-to-vertex ratio is reported
    for each configuration. Left: $\gamma_{\mathrm{med}} = 3 \times 10^4$.
    Right: $\gamma_{\mathrm{med}} = 3 \times 10^3$. $\kappa$ in the title is the full friction
    matrix condition number.
    {Bottom Left:} $n=50\,000$ cells randomly packed in a spheroidal domain with
    $\gamma_{\mathrm{med}} = 3 \times 10^4$.
    {Bottom Right:} Two densely packed spheroids connected by a bridge of randomly
    placed cells, $n \approx 50\,000$ total, $\gamma_{\mathrm{med}} = 3 \times 10^4$.
    }\label{fig:random_benchmark_v2}
\end{figure}

\begin{figure}[htbp]
    \includegraphics[width=0.9\textwidth]{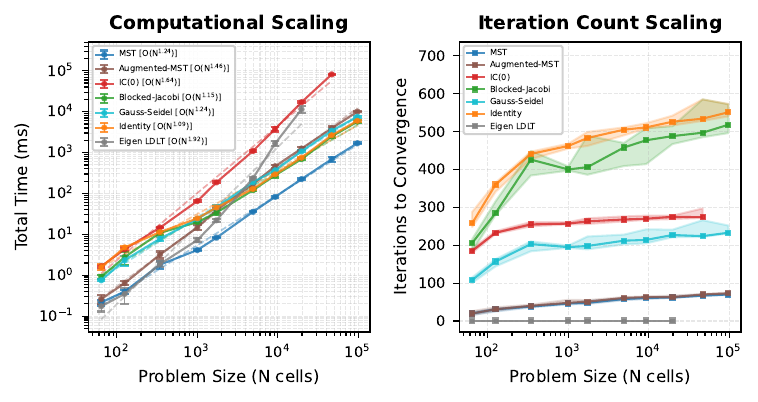}
    \caption{
        Performance comparison of preconditioners for block-structured linear
        systems arising from $n$
        cells arranged in a hexagonal lattice with positional noise
        (30\% of cell radius). Each data point represents the average of 5 independent experiments.
        The friction parameters are $\gamma_{\mathrm{med}} = 3 \times 10^4$ with
        $\gamma_{\parallel} = 2 \times 10^6$ and $\gamma_{\perp} = 8 \times 10^7$.
        {(Left)} Total wall-clock time as a function of problem size $n$ for
        solving the linear system to relative tolerance $10^{-5}$. Times include
        all computational costs: matrix assembly (where required), preconditioner
        construction and factorization, and conjugate gradient iterations.
        The MST preconditioner achieves a speedup of $3.5\times$ compared to unpreconditioned CG
        at $n = 10^5$. The direct solver (Eigen LDLT) excels for $n < 10^3$ but becomes catastrophically slow
        for larger problems. Preconditioners tested:
        (1) identity (no preconditioning), (2) block Jacobi with $LDL^T$ factorization of diagonal blocks,
        (3) MST with tree construction and factorization, (4) augmented-MST with tree construction,
        augmentation and nested (flexible) conjugate gradient (inner relative tolerance $\delta = 0.1$),
        (5) IC(0) with diagonal shift
        and incomplete factorization, (6) block Gauss-Seidel with setup, and
        (7) Eigen's sparse $LDL^T$ direct solver including assembly and factorization.
        The MST preconditioner achieves the fastest total solution times.
        {(Right)} Iteration counts for the same preconditioners as a function of $n$.
    }\label{fig:scaling}
\end{figure}

In the previous sections, we introduced the collision graph and block Laplacians, and we
extended support graph theory to block-structured
matrices, giving a class of preconditioners that can be contructed directly from the collision graph. Although we derived theoretical
bounds on the eigenvalues of the preconditioned system, past experience with
using the conjugate gradient method in finite precision arithmetic demonstrates that
theoretical estimates often do not accurately predict performance in practice \cite{liesen2013krylov}.
For this reason, we conduct benchmarking experiments to compare six preconditioning
strategies:
(1) no preconditioner,
(2) the block Jacobi preconditioner (i.e.,\ the block diagonals of $\Gamma$),
(3) the symmetric block Gauss-Seidel preconditioner,
(4) the MST preconditioner from Section~\ref{sec:vaidya},
(5) the augmented MST preconditioner from Section~\ref{sec:vaidya}, and
(6) the block incomplete Cholesky preconditioner with no edge fill-in together with a diagonal shift $\alpha I$.

Our implementation begins with collision detection using an axis-aligned bounding box
approach \cite{tracy2009efficient} that achieves $\mathcal{O}\lb n \log n\rb$ complexity with an
event-based algorithm. This produces an active set of edges representing potential
contacts, from which we compute contact forces, contact areas, and friction blocks according
to equation~\eqref{eq:cell_cell_friction} and the Hertz contact model. Our
implementation uses a sweep and prune (SAP) algorithm without hierarchical subdivision, which limits
our benchmarks to $n \leq 10^5$ cells where SAP remains efficient.

The resulting contact graph is converted to a block-structured BCSR (Block Compressed
Sparse Row) format for efficient matrix-vector products, with reverse Cuthill-McKee
(RCM) ordering to improve cache locality. The MST-preconditioner is solved using an
$LDL^T$ decomposition directly on the graph structure. The augmented MST uses nested conjugate gradient
iterations following the flexible CG framework \cite{Notay2000}. Since we observed no
iteration count improvements with augmentation, we did not pursue more efficient solution
methods for this variant. All preconditioner construction operations (graph building, tree
computation, and factorization) are cheaper than a single matrix-vector product, with
friction blocks assembled during force calculation at no additional cost.

The solver is implemented in C++ using an expression-based linear algebra system with custom block-structured vectors that leverage BLAS routines. The preconditioner construction process, comprising Prim's algorithm and tree assembly, is highly efficient, requiring less wall-clock time than a single matrix-vector product. After construction and factorization, solving with the MST preconditioner consumes 40–80\% of the time needed for one matrix-vector product. Notably, approximately one-third of this solving time is spent permuting matrix entries between two orderings: the RCM ordering used for matrix-vector operations and the ordering from Prim's algorithm used for factorization. This overhead suggests room for improvement through more sophisticated ordering strategies that simultaneously optimize both matrix-vector products and direct solves.

For our convergence benchmarks, we employ a rigorous testing framework. We work with known solutions to compute true relative errors and use Gauss-Radau error estimators \cite{Meurant2024} to monitor conjugate gradient convergence. To ensure fair comparisons across all benchmarks, we use single-threaded execution throughout.

\subsection{Experimental Setup}

To test the preconditioning strategies, we simulate elastic spherical cells of radius $r = 0.5$
(units are in 10s of $\mu m$) using the Hertz contact model for cell-cell contact areas and
forces. We assume that the cell-substrate friction matrix is diagonal (i.e., $\Gamma^{cs} = \gamma_{\mathrm{med}} I$),
and the cell-cell friction matrix is given by equation~\eqref{eq:cell_cell_friction} with
$\gamma_{\parallel} = 2 \times 10^6$ and $\gamma_{\perp} = 8 \times 10^7$, representing
typical values where cell-cell friction coefficients are 1--3 orders of magnitude greater
than cell-substrate friction \cite{Galle2005}.
We evaluate preconditioner performance across four distinct scenarios designed to test
different geometric configurations and conditioning challenges:

\begin{enumerate}
\item \textbf{Regular hexagonal lattice (baseline):} Cells arranged in hexagonal close
packing where each cell contacts 12 neighbors. We add positional noise of mean zero and
standard deviation $0.3r$ to create varying cell-cell contact
areas. We test with $\gamma_{\mathrm{med}} = 3 \times 10^4$
(well-conditioned) and $\gamma_{\mathrm{med}} = 3 \times 10^3$ (higher condition number).

\item \textbf{Random packing:}
Cells are randomly placed within a spheroidal domain using a rejection sampling algorithm that enforces a minimum distance constraint to prevent overlap. This process generates the irregular connectivity patterns characteristic of biological tissues, avoiding the artifacts of regular lattices.
For all simulations, we use $\gamma_{\mathrm{med}} = 3 \times 10^4$.

\item \textbf{Dumbbell configuration:}
Two densely packed spheroids connected by a thin
bridge of randomly placed cells. This tests preconditioner performance on graphs
with bottleneck structures. We use $\gamma_{\mathrm{med}} = 3 \times 10^4$.
\end{enumerate}

For the convergence behavior study (Figure~\ref{fig:random_benchmark_v2}),
we use $n = 50\,000$ cells in each configuration and track the relative
error reduction over conjugate gradient iterations. For the scaling study
(Figure~\ref{fig:scaling}), we use the regular lattice configuration
with $\gamma_{\mathrm{med}} = 3 \times 10^4$ and vary the problem size from
$n = 10^2$ to $10^5$ cells, measuring both iteration counts and wall-clock time.

All experiments use a relative tolerance of $10^{-5}$ for the conjugate
gradient stopping condition, with true error computed using a known solution vector.
Each data point represents the average of 5~independent random realizations.

\subsection{Benchmarking Results}

Figure~\ref{fig:random_benchmark_v2} presents the convergence behavior of different
preconditioners across four challenging geometric configurations with $n=50\,000$ cells.

The MST preconditioner demonstrates superior performance across all test cases. On hexagonal
lattices (top row), iteration reduction improves from $4.9\times$ at moderate condition number
to $9.2\times$ at high condition numbers, significantly outperforming block Jacobi ($1.2\times$)
and block Gauss-Seidel ($2.5\times$) for both moderate and high condition number scenarios.
Performance remains strong on irregular geometries:
$3.8\times$ reduction for random spheroidal packing and $8.5\times$ for the sparse dumbbell
configuration, where traditional preconditioners struggle.

The augmented variant shows negligible improvement, likely because standard spanning trees
already achieve low stretch for these graphs. Block Gauss-Seidel maintains consistent $2.5\times$
reduction across geometries, while diagonally shifted IC(0) degrades from $2.0\times$ to $1.3\times$
on sparse configurations.

Figure~\ref{fig:scaling} demonstrates computational scaling from $n = 10^2$ to $10^5$ cells.
At production scale ($n = 10^5$), the MST preconditioner achieves between $3.4\times$ and $4.4\times$ wall-clock
speedup over standard methods, with preconditioner construction and application each requiring
less effort than a single matrix-vector product. This efficiency, combined with favorable
iteration count scaling (right panel), ensures performance advantages increase with problem size.
Alternative methods scale poorly: the augmented MST preconditioner incurs overhead from nested iteration, while
incomplete Cholesky performs poorly and requires tuning of the diagonal shift parameter.

\section{Discussion}
\label{sec:discussion}

We have proposed efficient preconditioners for linear systems arising from off-lattice cell-based models.
Using the notions of matrix-weighted graphs and block Laplacians, we extended support graph theory to this problem.
By using MSTs, we obtained preconditioners that are efficient to compute and factor,
while significantly decreasing the condition number, iteration count, and wall-clock time in the conjugate gradient method.
We proved bounds on the condition number
for both standard and augmented MSTs that, while overly pessimistic in practice, demonstrate asymptotic stability.
Although augmented MST preconditioners have proven effective for mesh Laplacians, our implementation using
nested conjugate gradient iterations showed no iteration count improvement for cell-based models.
We hypothesize that the local interaction graphs from spherical cell packings already achieve sufficiently
low stretch without augmentation, though other cell geometries or packing configurations
might benefit from this technique.

We comprehensively benchmarked our proposed preconditioners against identity,
block Jacobi, block Gauss-Seidel, block IC(0), and direct sparse solvers across problem
sizes from $n=10^2$ to $10^5$. The MST preconditioner achieves the best wall-clock times for larger
systems relevant to production simulations while remaining competitive with direct solvers for
small systems.
The magnitude of performance improvement varies with problem characteristics, particularly condition number and graph topology,
rather than degrading systematically with connectivity. Indeed, our most substantial improvements occurred on challenging geometries:
high difference in parallel and perpendicular cell-cell friction coefficients and
near-disconnected configurations. This robust performance across diverse scenarios establishes the MST preconditioner
as consistently advantageous for off-lattice cell-based models.

For practical implementation, several key insights emerge. The collision graph, naturally produced by collision detection algorithms, serves
dual purposes: defining the system's friction matrix and enabling direct construction of the MST preconditioner. This graph representation allows
elegant implementations of the required linear algebra operations without intermediate matrix assembly. Construction of the MST preconditioner, including graph traversal, tree computation, and factorization, requires less computational effort than a single matrix-vector product,
with friction blocks assembled during force calculation at no additional cost. The preconditioner solve requires only 40--80\% of the time for one matrix-vector product,
ensuring that iteration reductions translate directly to wall-clock performance gains.

This paper generalizes and applies the earliest results in support graph theory. While Vaidya's preconditioners are very efficient to compute, more sophisticated preconditioners achieve much better condition numbers and near-linear time convergence. The main tool they employ is the low-stretch spanning tree. These are the basis of Spielman's groundbreaking paper \cite{spielman2014nearly} that uses augmented low-stretch spanning trees as preconditioners. We can generalize the notion of stretch to matrix-weighted graphs by examining the minimum eigenvalues of the edge weights, similar to our approach for MSTs. Substantial progress has been made in solving symmetric diagonally dominant systems with preconditioners derived from low-stretch spanning trees, most recently in \cite{jambulapati2021} and \cite{gao2023robust}. Additionally, with a low-stretch spanning tree, we can implement combinatorial algorithms like those described in \cite{kelner2013} and \cite{lee2013efficient} that do not use the conjugate gradient method at all. We expect that all results from the existing literature can be generalized to block-structured matrices by adding a factor of $\kappa$ to the condition number bounds.

Future work will exploit the hierarchical spatial data structures already required for collision detection in large simulations ($n > 10^5$) where SAP becomes inefficient without spatial subdivision \cite{tracy2009efficient}.
These structures (octrees, Z-order, or Hilbert curve grids) partition cells into clusters of 300–1000 elements to maintain optimal SAP performance at each leaf.
We envision constructing a hierarchical preconditioner where each leaf node maintains its own MST preconditioner for the local interaction subgraph. The full preconditioner consists of these independent local preconditioners applied in parallel, simply ignoring any interactions between regions. While this drops some edges from the preconditioner,
it enables perfect parallelization with no inter-processor communication.
The key insight is that the spatial partitioning required to maintain SAP efficiency also creates
natural subdomains for embarrassingly parallel preconditioning, addressing both  collision detection
and linear solving bottlenecks simultaneously through a unified hierarchical approach.

\appendix

\section{Positive Definite Off-Diagonal Blocks}
\label{app:posdefoff}

An ostensible limitation of the theory presented in this paper is that it only works with block Laplacians, requiring off-diagonal blocks to be zero or negative definite. This section describes methods of working with positive definite off-diagonal blocks as well.

There are two established ways of handling positive off-diagonal entries in the non-block case. The first is a reduction due to Gremban \cite[Lemma 7.3]{Gremban1996} which is also described concisely in \cite[Appendix A]{spielman2014nearly} and \cite[Appendix A.2]{maggs2005finding}. The validity of this reduction immediately transfers to the block case. Let $A$ be a block matrix whose off-diagonal blocks are either zero or definite (positive or negative), and for every row $i$,
\[
    \mat{A_{ii}} \succeq \sum_{j \neq i} |\mat{A_{ij}}|
\]
where $|\cdot|$ leaves positive semidefinite matrices unchanged and negates negative semidefinite ones. This matrix is positive semidefinite by the proof from Lemma \ref{blapposdef}. We may call this a generalized block Laplacian, possibly originating from a similarly defined generalized matrix-weighted graph. Then $A = D + A^{(+)} + A^{(-)}$ where $D$ contains the diagonal blocks, $A^{(+)}$ the positive definite off-diagonal blocks, and $A^{(-)}$ the negative definite ones. To solve the system $A\vec{x} = \vec{b}$, we construct a $2n \times 2n$ block Laplacian system
\[
    A'\begin{pmatrix}
        \vec{x}_1 \\ \vec{x}_2
    \end{pmatrix} = \begin{pmatrix}
        \vec{b} \\ -\vec{b}
    \end{pmatrix} \quad \text{where} \quad A' = \begin{pmatrix}
        D + A^{(-)} & -A^{(+)} \\
        -A^{(+)} & D + A^{(-)}
    \end{pmatrix}.
\]
The desired solution is then $\vec{x} = (\vec{x}_1 - \vec{x}_2)/2$. Thanks to the simplicity of this reduction, the condition number analysis from earlier still applies and $\epsilon$-approximate solutions to the larger system produce $\epsilon$-approximate solutions to the original one.

The other method of solving a generalized Laplacian system is with a maximum weight basis preconditioner \cite{Boman2004}. This is a considerably more complicated technique that requires additional analysis of the condition number, but it does not require a reduction to a larger problem. We have not proven any analogues for the block case.

\section{Assorted Graph Algorithms}

\subsection{Directly Solving Systems from Trees}
\label{sec:treesolve}

Once the $LDL^T$ decomposition of an MST preconditioner $P$ is computed, we need to repeatedly solve systems $P\vec{x} = \vec{b}$. This section describes how to do this in terms of the implicit tree structure of $P$ (i.e., we do not distinguish between the indices of rows and columns and the vertices they represent). We assume the rows and columns of $P$ are ordered as described in Lemma \ref{nofill} so that $L$ has the same sparsity pattern as the lower triangle of $P$. Let $d$ be the order of the blocks of $P$.

\begin{enumerate}[label=\textbf{Step~\arabic*:}, ref=Step~\arabic*, leftmargin=*, labelindent=\parindent]
    \item
    \underline{Forward substitution.}
        Define $\vec{z} = DL^T \vec{x}$ and solve $L\vec{z} = \vec{b}$, and proceed as in
        algorithm~\ref{alg:fwdsubst}.
        The cost is $n - 1$ matrix-vector multiplications of $d \times d$ matrices.

\begin{algorithm}[H]
    \caption{}\label{alg:fwdsubst}
    \begin{algorithmic}
    \Function{ForwardSolve}{$L$, $\vec{b}$}
    \For{$i \gets 1, \ldots, n$}
        \State $\mat{\vec{z}_i} \gets \mat{\vec{b}_i}$
        \ForAll{children $j$ of $i$}
            \State $\mat{\vec{z}_i} \gets \mat{\vec{z}_i} - \mat{L_{ji}\vec{z}_j}$
        \EndFor
    \EndFor
    \State \Return $\vec{z}$
    \EndFunction
    \end{algorithmic}
\end{algorithm}

    \item
    \underline{Block diagonal solve.}
        Define $\vec{y} = L^T \vec{x}$ and solve $D\vec{y} = \vec{z}$.
        The cost is $n$ solves of $d \times d$ matrices.

\begin{algorithm}[H]
    \caption{}\label{alg:diagsolve}
    \begin{algorithmic}
    \Function{DiagSolve}{$D$, $\vec{z}$}
    \For{$i \gets 1, \ldots, n$}
        \State $\mat{\vec{y}_i} \gets \mat{D_{ii}^{-1} \vec{z}_i}$
    \EndFor
    \State \Return $\vec{y}$
    \EndFunction
    \end{algorithmic}
\end{algorithm}

    \item
    \underline{Backward substitution.}
    Solve $L^T \vec{x} = \vec{y}$.
    The cost is $n - 1$ matrix-vector multiplications of $d \times d$ matrices.

\begin{algorithm}[H]
    \caption{}\label{alg:bwdsubst}
    \begin{algorithmic}
    \Function{BackwardSolve}{$L$, $\vec{y}$}
    \For{$i \gets n, \ldots, 1$}
        \State $j \gets \text{parent of }i$
        \State $\mat{\vec{x}_i} \gets \mat{\vec{y}_i} - \mat{L_{ij}^T \vec{x}_j}$
    \EndFor
    \State \Return $x$
    \EndFunction
    \end{algorithmic}
\end{algorithm}

\end{enumerate}

\subsection{Matrix-Vector Product of a graph}
\label{subsec:matvec}

Lastly, we give an algorithm for computing matrix-vector products with the block Laplacian of a matrix-weighted graph. This is needed for a matrix-free implementation of the conjugate gradient method.

\begin{algorithm}[H]
    \caption{}
    \begin{algorithmic}
    \Function{MatVec}{$G = (V, E, w)$, $\vec{v}$}

    \For{$i \gets 1, \ldots n$}
        \State $\mat{\vec{x}_i} \gets w(i, i) \mat{\vec{v}_i}$
    \EndFor

    \For{$(i, j) \in E$}
        \State $\mat{\vec{x}_i} \gets \mat{\vec{x}_i} + w(i, j)(\mat{\vec{v}_i} - \mat{\vec{v}_j})$
        \State $\mat{\vec{x}_j} \gets \mat{\vec{x}_j} + w(i, j)(\mat{\vec{v}_j} - \mat{\vec{v}_i})$
    \EndFor
    \EndFunction
    \end{algorithmic}
\end{algorithm}

\bibliographystyle{amsplain}

\end{document}